\documentclass[11pt,letterpaper]{amsart}

\usepackage{amsthm}
\usepackage{amsmath,amssymb,amsfonts,amsthm,color,latexsym}
\usepackage{xcolor}
\usepackage{float}
\usepackage{booktabs}
\newcommand{\C}{\mathbb{C}}
\newcommand{\F}{\mathcal{F}}

\newcommand{\m}{\mathfrak{m}}
\newtheorem{theorem}{Theorem}[section]

\usepackage[pagebackref]{hyperref}

\newcommand{\cpt}[1]{\mathbb{C}P^{2}}

\newcommand{\B}{\mathcal{B}}

\newcommand{\sep}{\mbox{\rm Sep}}

\newtheorem{lemma}[theorem]{Lemma}
\newtheorem{proposition}[theorem]{Proposition}
\newtheorem{corollary}[theorem]{Corollary}
\newtheorem{definition}{Definition}[section]
\newtheorem{example}{Example}[section]
\newtheorem{remark}{Remark}[section]
\title{On the $k$-th Milnor and $k$-th Tjurina Numbers of a Foliation}
\date{\today}

\author{Marcela Ribeiro}
\address[Marcela Ribeiro]{Department of Mathematics. Federal University of Minas Gerais. Av. Pres. Ant\^onio Carlos, 6627 
CEP 31270-901\\
Pampulha - Belo Horizonte - Brazil}
\email{marcelaribeiro9719@gmail.com}

\author[A. Fern\'andez-P\'erez]{Arturo Fern\'andez-P\'erez}
\address[Arturo Fern\'{a}ndez P\'erez] {Department of Mathematics. Federal University of Minas Gerais. Av. Pres. Ant\^onio Carlos, 6627 
CEP 31270-901\\
Pampulha - Belo Horizonte - Brazil. ORCID: 0000-0002-5827-8828\\}
\email{fernandez@ufmg.br}

\thanks{The first author acknowledges support from CAPES-Brazil. The second author acknowledges support from CNPq Projeto Universal 408687/2023-1 ``Geometria das Equa\c{c}\~oes Diferenciais Alg\'ebricas" and the PQ fellowship CNPq-Brazil 306011/2023-9.}
\begin{document}
%\linenumbers

\begin{abstract}
In this paper, we introduce the notions of the $k$-th Milnor number and the $k$-th Tjurina number for a germ of holomorphic foliation on the complex plane with an isolated singularity at the origin. We develop a detailed study of these invariants, establishing explicit formulas and relating them to other indices associated with holomorphic foliations. In particular, we obtain an explicit expression for the $k$-th Milnor number of a foliation and, as a consequence, a formula for the $k$-th Milnor number of a holomorphic function. We analyze their topological behavior, proving that the $k$-th Milnor number of a holomorphic function is a topological invariant, whereas the $k$-th Tjurina number is not. In dimension two, we provide a positive answer to a conjecture posed by Hussain, Liu, Yau, and Zuo concerning a sharp lower bound for the $k$-th Tjurina number of a weighted homogeneous polynomial. We also present a counterexample to another conjecture of Hussain, Yau, and Zuo regarding the ratio between these invariants. Moreover, we establish a fundamental relation linking the $k$-th Tjurina numbers of a foliation and of an invariant curve via the G\'omez-Mont–Seade–Verjovsky index, and we extend Teissier’s Lemma to the setting of $k$-th polar intersection numbers. In addition, we derive an upper bound for the $k$-th Milnor number of a foliation in terms of its $k$-th Tjurina number along balanced divisors of separatrices. Finally, for non-dicritical quasi-homogeneous foliations, we obtain a closed formula for their $k$-th Milnor and Tjurina numbers.
\end{abstract}

\maketitle
\tableofcontents

\section{Introduction}

In recent years, new invariants associated with germs of complex hypersurfaces at $0 \in \C^n$ have been introduced with the aim of developing a classification theory of singularities. Among them, we highlight the so-called \textit{$k$-th Milnor and Tjurina numbers of a hypersurface}, which are the respective dimensions of the \textit{$k$-th Milnor and Tjurina algebras} defined by G.-M.~Greuel and T.~H.~Pham \cite{Pham}. 
\par More precisely, let $\mathcal{O}_n = \C\{x_1, \ldots, x_n\}$ denote the local ring of convergent power series, and let $\m$ be its maximal ideal. For a hypersurface with an isolated singularity $(V,0) \subset (\C^n,0)$, where $V=V(f)=\{f=0\}$ and $f\in\mathcal{O}_n$, the \textit{$k$-th Milnor number} and the \textit{$k$-th Tjurina number} of $V$ are defined, respectively, as follows:
\[
\mu^{k}(V) := \dim \frac{\mathcal{O}_n}{\m^k j(f)}, \qquad 
\tau^{k}(V) := \dim \frac{\mathcal{O}_n}{\langle \m^k j(f), f \rangle}, 
\qquad \text{for every integer } k \geq 0,
\]
where $j(f) = \langle f_{x_1}, \ldots, f_{x_n} \rangle$ denotes the Jacobian ideal of $f$.  
\par It follows from the Mather–Yau Theorem (see, for instance, \cite[Theorems 2.26 and 2.28]{Greuel}) that $\mu^k(V)$ and $\tau^k(V)$ are \textit{analytic invariants}. 
Recently, Hussain, Liu, Yau, and Zuo \cite{k-Milnor} investigated the $k$-th Milnor and $k$-th Tjurina numbers of \textit{weighted homogeneous singularities}. They proved that, when $V$ is a binomial plane curve singularity, both $\mu^k(V)$ and $\tau^k(V)$ are topological invariants. Moreover, the authors explicitly computed these invariants for binomial singularities; see Theorems A, B, and C in \cite{k-Milnor}. The paper concludes with an analytic characterization of simple hypersurface singularities via the so-called \textit{$k$-th Yau algebras}; see Theorem D in \cite{k-Milnor}. For a study on the $k$-th Tjurina number and the dimension of the semi-universal deformation space, see \cite{Tjurina 2}.

\par Motivated by these definitions and results, the purpose of this paper is to define the \textit{$k$-th Milnor number} and the \textit{$k$-th Tjurina number} of a germ of holomorphic foliation $\F$ defined by a holomorphic $1$-form $\omega \in \Omega^{1}(\C^2,0)$ with an isolated singularity at $0 \in \C^2$, as follows.  

Let $\omega = P(x,y)\,dx + Q(x,y)\,dy$ be a germ of holomorphic $1$-form at the origin of $\C^2$ such that $\F$ is defined by $\omega$ and has an isolated singularity at $0 \in \C^2$. Let $C  = \{ f = 0 \}$ be the germ of a singular complex curve, where $f \in \mathcal{O}_2$. We say that $C$ is \textit{invariant} by $\F$ if $(\omega \wedge df)/f \in \Omega^{2}(\C^2,0)$. Then, we define the \textit{$k$-th Milnor number} of $\F$ at the origin as
\[
\mu^{k}(\F) := \dim \frac{\mathcal{O}_2}{\langle P, Q \rangle \m^k},
\]
and the \textit{$k$-th Tjurina number} of $\F$ along $C$ as
\[
\tau^{k}(\F,C) := \dim \frac{\mathcal{O}_2}{\langle P, Q \rangle \m^k + \langle f \rangle}.
\]

\par Note that when $k = 0$, we have $\mu^{0}(\F) := \mu(\F)$ and $\tau^{0}(\F,C) := \tau(\F,C)$, which coincide with the classical Milnor number of $\F$, defined by Camacho–Lins Neto–Sad \cite{CLS}, and the Tjurina number of $\F$ along $C$, introduced by Cano–Corral–Mol \cite{Cano}, respectively. Also observe that when $\F$ is given by a holomorphic first integral $\omega = df$, we have $\mu^{k}(\F) = \mu^{k}(C)$ and $\tau^{k}(\F,C) = \tau^{k}(C)$.

\par Naturally, our definition extends to germs of holomorphic $1$-forms with an isolated singularity at the origin in $\C^n$, as follows. Let $\omega = \sum_{i=1}^{n} P_i\,dx_i$ and $V = \{ f = 0 \}$ be a hypersurface invariant by $\omega$. Then, for every integer $k \geq 0$, we define the $k$-th Milnor number of $\omega$ by
\[
\mu^{k}(\omega) = \dim \frac{\mathcal{O}_n}{\langle P_1, \ldots, P_n \rangle \m^k},
\]
and the $k$-th Tjurina number of $\omega$ with respect to $V$ as
\[
\tau^{k}(\omega,V) = \dim \frac{\mathcal{O}_n}{\langle P_1, \ldots, P_n \rangle \m^k + \langle f \rangle}.
\]

\par In this paper, we mainly study the case $n = 2$, that is, the $k$-th Milnor and Tjurina numbers of a germ of foliation at the origin of $\C^2$. We begin the paper by recalling basic definitions of foliation theory in Section~\ref{basic_fol}, and in Section~\ref{co_dimension} we prove results on colengths of ideals in $\mathcal{O}_2$ that will be useful for the proofs of our main results.

As a first result, in Section~\ref{fomula_Milnor}, we obtain an explicit formula for the $k$-th Milnor number of a foliation $\F$ (see~\eqref{eq_apli}):
\[
\mu^{k}(\F) = \mu(\F) + \dim \frac{\mathcal{O}_2}{\m^{k}} + \dim \frac{\mathcal{O}_2}{\langle P \rangle + \m^{k}},
\]
where $\F$ is defined by $\omega=Pdx+Qdy$ and the algebraic multiplicity of $\F$ satisfies $\nu(\F)=\nu(P)$. As a consequence, we obtain an explicit formula for $\mu^{k}(f)$ (see Corollary~4.1).  
Combining this result with \cite[Proposition~2.6]{k-Milnor}, we recover Theorems~A and~B from the same reference.  
In Section~5, we show that the $k$-th Milnor number of a germ $f \in \mathcal{O}_2$ is a topological invariant, whereas the $k$-th Tjurina number is not.  
\par In Section~6, we provide a positive answer to a conjecture posed by Hussain, Liu, Yau, and Zuo (see \cite[Conjecture~1.1]{k-Milnor}) concerning a sharp lower bound for the $k$-th Tjurina number of a weighted homogeneous polynomial $f \in \C[x,y]$.
 Section~\ref{Conjecture} is dedicated to a conjecture proposed in \cite[Conjecture~1.3]{Hussain_1}, which states that
\[
\frac{\mu^{k}(f)}{\tau^{k}(f)} < \frac{4}{3} \quad \forall \, k \geq 0.
\]
Using \texttt{Singular} \cite{Singular}, we provide a negative answer to this conjecture by constructing a counterexample (see Equation~\eqref{ex_conjectura}). We also show that the conjecture holds under a certain restriction on $k$ (see Theorem~\ref{teo_conje}).  

\par In Section~\ref{GSV-section}, we obtain the following equality:
\[
\tau^k(\F,C) - \tau^k(C) = GSV(\F,C),
\]
where $GSV(\F,C)$ denotes the G\'omez-Mont–Seade–Verjovsky index of $\F$ along $C$ at the origin (see \cite{GSV}).  

Next, in Section~\ref{polar_section}, we define the \textit{$k$-th polar intersection number} of a foliation $\F$ with respect to an invariant curve $C = \{f = 0\}$. As an application, we prove an extension of Teissier’s Lemma \cite{Teissier}, providing a formula relating this $k$-th index with $\mu^{k}(f)$.  

\par In Section~\ref{bound_section}, combining ideas from \cite{Liu} and \cite{Hussain_1}, we present the following inequality:
\[
\tau^{k}(\F,\B_0) \leq \mu^{k}(\F) \leq 2\tau^{k}(\F,\B_0) + \frac{k(k+1)}{2},
\]
where $\B_0$ denotes the divisor of zeros of a balanced divisor of separatrices $\B$ for $\F$. Such divisors were introduced by Genzmer \cite{Genzmer} and play an important role in the study of dicritical foliations.  

\par Finally, in Section~\ref{section_quasi}, using a characterization of quasi-homogeneous foliations due to Mattei \cite{quasiMattei}, we obtain the following formula:
\[
\mu^{k}(\mathcal{F}) = \tau^{k}(\mathcal{F}, C) + \frac{k(k-1)}{2}, \qquad \forall \, k \geq 1.
\]

\par The following notations will be used in this paper:
\begin{enumerate}
\item $\mathcal{O}_2$: the ring of germs of holomorphic functions at $0\in\C^2$.
\item $\m=\{ f\in\mathcal{O}_2 : f(0)=0\}$, the maximal ideal of $\mathcal{O}_2$.
\item $\nu(f)$: algebraic multiplicity of $f\in\mathcal{O}_2$.
\item $\mu(f)$: Milnor number of $f\in\mathcal{O}_2$.
\item $\tau(f)$: Tjurina number of $f\in\mathcal{O}_2$.
\item $\mu^k(f)$: $k$-th Milnor number of $f\in\mathcal{O}_2$.
\item $\tau^k(f)$: $k$-th Tjurina number of $f\in\mathcal{O}_2$.
\item $\F$: germ at $0\in\C^2$ of singular foliation.
\item $\nu(\F)$: algebraic multiplicity of $\F$ at $0\in\C^2$.
\item $\mu(\F)$: Milnor number of $\F$.
\item $\tau(\F,C)$: Tjurina number of $\F$ with respect to $C$ at $0\in\C^2$.
\item $\mu^{k}(\F)$: $k$-th Milnor number of $\F$.
\item $\tau^{k}(\F,C)$: $k$-th Tjurina number of $\F$ with respect to $C$ at $0\in\C^2$.
\item $i(f,g)$: intersection number of germs $f$ and $g$. 
\end{enumerate}
\section{Basic notions on foliations}\label{basic_fol}

\par  To establish the terminology and notation, we recall  some basic notions of the theory of holomorphic foliations. Unless stated otherwise, throughout this text, $\F$ denotes a germ of a singular holomorphic foliation at $(\C^2,0)$. In local coordinates $(x,y)$ centered at $0\in\C^2$, the foliation $\F$ is given by a germ of a holomorphic 1-form
\begin{equation*}
\omega=P(x,y)dx+Q(x,y)dy,
\end{equation*}
or by its dual vector field
\begin{equation}\label{oneform}
v = -Q(x,y)\frac{\partial}{\partial{x}} + P(x,y)\frac{\partial}{\partial{y}},
\end{equation}
where  $P(x,y), Q(x,y)   \in \mathcal{O}_2$ are relatively prime. 
The \textit{algebraic multiplicity} $\nu(\F)$ is the minimum of the algebraic multiplicities $\nu(P)$, $\nu(Q)$ at $0\in\C^2$ of the coefficients of a local generator of $\F$.
\par Let $f(x,y)\in  \mathbb{C}[[x,y]]$, where ${\mathbb C}[[x,y]]$ is the ring of complex formal power series in two variables.  
We say that $C=\{ f(x,y)=0\}$  is {\em invariant} by $\F$ or $\F$-\emph{invariant} if $$\omega \wedge d f=(f.h) dx \wedge dy,$$ for some $h\in \mathbb{C}[[x,y]]$. If $C$ is an irreducible $\F$-invariant curve then we will say that $C$ is a {\em separatrix} of $\F$ at $0$. The separatrix $C$ is analytical if $f$ is convergent.  We denote by $\sep_0(\F)$ the set of all separatrices of $\F$ at $0$. When $\sep_0(\F)$ is a finite set we will say that the foliation $\F$ is {\em non-dicritical} and the union of all elements of $\sep_0(\F)$ is called {\em total union of separatrices} of $\F$ at $0$. Otherwise, we will say that $\F$ is a {\em dicritical} foliation.
\par We say that $0\in\C^2$ is a \textit{reduced} singularity for $\F$ if the linear part $\text{D}v(0)$ of the vector field $v$ in (\ref{oneform}) is non-zero and has eigenvalues $\lambda_1,\lambda_2\in\C$ fitting in one of the cases:
\begin{enumerate}
\item[(i)] $\lambda_1\lambda_2\neq 0$ and $\lambda_1\lambda_2\not\in\mathbb{Q}^{+}$ (\textit{non-degenerate});
\item[(ii)] $\lambda_1\neq 0$ and $\lambda_2= 0$ (\textit{saddle-node singularity}).
\end{enumerate}
In the case $(i)$, there is a system of coordinates  $(x,y)$ in which $\F$ is defined by the equation
\begin{equation}
\label{non-degenerate}
\omega=x(\lambda_1+a(x,y))dy-y(\lambda_2+b(x,y))dx,
\end{equation}
where $a(x,y),b(x,y)  \in {\mathbb C}[[x,y]]$ are non-units, so that  $\sep_0(\F)$ is formed by two
transversal analytic branches given by $\{x=0\}$ and $\{y=0\}$. In the case $(ii)$, up to a formal change of coordinates, the  saddle-node singularity is given by a 1-form of the type
\begin{equation}
\label{saddle-node-formal}
\omega = x^{\ell+1} dy-y(1 + \lambda x^{\ell})dx,
\end{equation}
where $\lambda \in \mathbb{C}$ and $\ell \in \mathbb{Z}^{+}$ are invariants after formal changes of coordinates (see \cite[Proposition 4.3]{martinetramis}).
The curve $\{x=0\}$   is an analytic separatrix, called {\em strong} separatrix, whereas $\{y=0\}$  corresponds to a possibly formal separatrix, called {\em weak} separatrix. The integer $\ell+1$ is called the  \textit{tangency index} of $\F$ with respect to the weak separatrix. 

\par  A germ divisor 
\[\displaystyle \B=\sum\limits_{B\in \operatorname{Sep}(\mathcal{F})} a_{B}\, B,\quad a_B\in\{-1,0,1\}\] is called \textit{balanced} if it satisfies the following conditions:
\begin{enumerate}
\item if $B$ is an isolated separatrix of $\F$ then $a_B=1$, 
\item if $a_B=-1$ then $B$ is a non-isolated (dicritical) separatrix of $\F$,
\item for each non-invariant (dicritical) component $D$ of the exceptional divisor $E$ of the reduction $\pi$ of singularities of $\F$ we have $\sum_{B}a_B \overline{B}\cdot D=2-val_D$, where $val_D$ is the valence of $D$, i.e. the number of irreducible components of $\overline{E\setminus D}$ meeting $D$.
\end{enumerate}
Notice that a non-dicritical foliation has a unique balanced divisor which is the sum of the isolated separatrix. In contrast, a dicritical foliation have infinitely many balanced divisors, see \cite{Genzmer}.
\par Recall that a foliation is said to be of \textit{second type} if none of the singularities of $\pi^*\mathcal{F}$ are tangent saddle-nodes; that is, whenever saddle-nodes appear in the reduction, their strong separatrix is transverse to the exceptional divisor. If no saddle-nodes appear in the reduction we say that $\F$ is a \textit{generalized curve}.
The balanced divisors characterize second type foliations:
\begin{proposition}[\cite{Genzmer}, Proposition 2.4]\label{prop:Equa-Ba}
Let $\F$ be a foliation germ on $(\C^2,0)$ with balanced divisor $\B$. Then $\F$ is of second type if and only if $\nu(\F)=\nu(\B)-1$.
\end{proposition}

\par We now proceed to compute the $k$-th Milnor and Tjurina numbers of a foliation with a reduced singularity at the origin. Since in this case the coordinates are formal, we will work with the definitions in the ring $\C[[x,y]]$.
\begin{proposition}
Let $\F$ be a germ of a singular foliation at $(\C^2,0)$. The following holds:  
\begin{enumerate}
    \item If $0\in\C^2$ is a non-degenerate singularity, then 
    \[\mu^{k}(\mathcal{F})= \frac{(k+1)(k+2)}{2},\qquad\qquad \tau^{k} (\mathcal{F}, \sep_0(\F))  = 2k + 1.\]
    \item If $0\in\C^2$ is a saddle-node singularity and $\ell+1$ is the tangency index of $\F$ with respect to the weak separatrix, then  \[\mu^{k}(\mathcal{F}) = \frac{(k+1)(k+2)}{2}  + \ell,\qquad \tau^{k} (\mathcal{F}, \sep_0(\F)) =  2k + 1 + \ell.\]
\end{enumerate}
\end{proposition}
\begin{proof}
First, assume that $0 \in \C^2$ is a non-degenerate singularity. Then, by \ref{non-degenerate}, the foliation $\mathcal{F}$ is given by
$\omega = x(\lambda_1 + a(x,y))dy - y(\lambda_2 + b(x,y))dx,$
where $a(x,y), b(x,y) \in {\mathbb C}[[x,y]]$ are non-units.
Then
\begin{eqnarray*}
\mu^{k}(\mathcal{F}) &=& \dim \frac{\C[[x,y]]}{ \langle-y(\lambda_2 + b(x,y)), x(\lambda_1 + a(x,y))\rangle\m^{k}}\\
&=& \dim  \frac{\C[[x,y]]}{\m^{k+1}}= \frac{(k+1)(k+2)}{2}.
\end{eqnarray*}

Now, the $k$-th Tjurina number of a foliation $\mathcal{F}$ along $\sep_0(\F) = \{xy = 0\}$ is given by
\begin{eqnarray*}
\tau^{k}(\mathcal{F},\sep_0(\F)) &=& \dim \frac{\C[[x,y]]}{ \langle-y(\lambda_2 + b(x,y)), x(\lambda_1 + a(x,y))\rangle\m^{k} + \langle xy\rangle}\\
&=& \dim\frac{\C[[x,y]]}{\langle x^{k+1}, y^{k+1}, xy\rangle}\\
&=& 2k + 1.
\end{eqnarray*}
 This completes the proof of item (1).

\par For item (2), we assume that $0 \in \C^2$ is a saddle-node singularity. Then, by (\ref{saddle-node-formal}), we have that $\F$ is given by the $1$-form
\begin{equation*}
\omega = x^{\ell+1} dy - y(1 + \lambda x^{\ell})dx,
\end{equation*}
where $\lambda \in \mathbb{C}$ and $\ell \in \mathbb{Z}^{+}$. Then
\begin{eqnarray*}
\mu^{k}(\mathcal{F})&=& \dim\frac{\C[[x,y]]}{\langle -y(1 + \lambda x^{\ell}), x^{\ell+1}\rangle \m^{k}}\\
&=& \dim\frac{\C[[x,y]]}{\langle -y, x^{\ell+1}\rangle\m^k}\\
&=& \dim\frac{\C[[x,y]]}{\langle x^{\ell+1+k}, x^{k}y, x^{k-1}y^{2}, \ldots, xy^{k}, y^{k+1}\rangle}\\
&=& \frac{(k+1)(k+2)}{2} + \ell.
\end{eqnarray*}

Finally, for the $k$-th Tjurina number of $\F$, we have
\begin{eqnarray*}
\tau^{k}(\mathcal{F}, \sep_0(\F)) &=& \dim\frac{\C[[x,y]]}{\langle -y(1 + \lambda x^{\ell}), x^{\ell+1}\rangle\m^{k} + \langle xy\rangle}\\
&=& \dim\frac{\C[[x,y]]}{\langle -y, x^{\ell+1}\rangle \m^{k} + \langle xy\rangle}\\
&=& \dim\frac{\C[[x,y]]}{\langle xy, x^{k+\ell+1}, y^{k+1}\rangle}\\
&=& 2k + 1 + \ell.
\end{eqnarray*}
\end{proof}
\begin{remark}
Let $\mathcal{F}$ be the germ of a foliation at $0\in\C^2$ with algebraic multiplicity $\nu(\mathcal{F})=m$. Then, for any integer $k\geq 0$, we have
$$
\mu^k(\mathcal{F}) \geq \frac{(m+k)(m+k+1)}{2}.
$$
Indeed, we have $( P,Q)\subset \mathfrak{m}^{m}$, which implies that $ \mathfrak{m}^k( P,Q) \subset \mathfrak{m}^{m+k}$. Therefore,
$$
\mu^k(\mathcal{F}) \geq \dim\left(\frac{\mathcal{O}_2}{\mathfrak{m}^{m+k}}\right) = \frac{(m+k)(m+k+1)}{2}.
$$
\end{remark}

\section{Colengths of ideals in $\mathcal{O}_2$}\label{co_dimension}
In this section, we present the commutative algebra lemmas that will be important to prove the results of this paper. 
We begin by establishing the following lemma:
\begin{lemma}\label{lemma_gsv 1}
    Let $J$ be an ideal in $\mathcal{O}_2$, and let $P,Q,f,g\in\mathcal{O}_2$ such that $f$ and $g$ are relatively prime. Suppose that the dimensions below are finite. Then
    \[\dim\frac{\mathcal{O}_2}{\langle gP,gQ\rangle\cdot J+\langle f\rangle}=\dim\frac{\mathcal{O}_2}{\langle P,Q\rangle\cdot J+\langle f\rangle}+\dim\frac{\mathcal{O}_2}{\langle g,f\rangle}.\]
\end{lemma}
\begin{proof}
The proof follows from the following short exact sequence:
\begin{align*}
0 \longrightarrow \frac{\mathcal{O}_2}{\langle P,Q\rangle\cdot J+\langle f\rangle} \xrightarrow{ \ \sigma  \ } \ \frac{\mathcal{O}_2}{\langle gP,gQ\rangle\cdot J+\langle f\rangle} \ \xrightarrow{ \ \delta  \ } \frac{\mathcal{O}_2}{\langle g,f\rangle} \longrightarrow 0,
\end{align*}
where $\sigma$ is the multiplication homomorphism by $g$, and $\delta$ is the natural projection homomorphism induced by the inclusion $\langle gP,gQ\rangle\cdot J+\langle f\rangle\subset \langle g,f\rangle$.
The injectivity of $\sigma$ follows from the fact that $f$ and $g$ are relatively prime. Indeed, if $gu=ga+fv$, with $a\in\langle P,Q\rangle\cdot J$, then $f$ divides $u-a$, and hence $u\in\langle P,Q\rangle\cdot J+\langle f\rangle$.
\end{proof}
\par Now, we present the following:
\begin{lemma}\label{lema do gsv}
  Let $J$  be an ideal in $\mathcal{O}_2$, and let $f,g\in\mathcal{O}_2$ such that $(f,g)=1$, i.e., they are relatively prime. Suppose that the dimensions below are finite. Then
    \begin{equation}
    \nonumber
        \dim \frac{\mathcal{O}_{2}}{\langle f\rangle \cdot J +\langle g \rangle}= \dim \frac{\mathcal{O}_{2}}{J+\langle g \rangle} + i(f,g),
    \end{equation}  
where $i(f,g)=\dim\frac{\mathcal{O}_2}{\langle f,g\rangle}$ is the intersection number of $f$ and $g$. 
\end{lemma}
\begin{proof}
     The proof is a consequence of the short exact sequence below:
\begin{align*}
0 \longrightarrow \frac{\mathcal{O}_{2}}{J +\langle g \rangle}  \xrightarrow{ \ \sigma  \ } \    \frac{\mathcal{O}_{2}}{\langle f\rangle \cdot J +\langle g\rangle} \xrightarrow{ \ \pi  \ } \frac{\mathcal{O}_{2}}{\langle f ,g \rangle} \longrightarrow 0,
\end{align*}
where $\sigma$ is the multiplication homomorphism by $f$, and $\pi$ is the natural projection homomorphism induced by the inclusion $\langle f\rangle \cdot J +\langle g \rangle \subset \langle f,g \rangle $.
\end{proof}
Now, we establish the following lemma:
\begin{lemma}\label{ca}
  Let $J=\m^k$, where $k\geq 0$ is an integer, and let $\psi,\varphi\in\m$ be nonzero and relatively prime. Suppose that $\nu(\psi)\leq \nu(\varphi)$, then
    \begin{equation}
    \nonumber
        \dim \frac{\mathcal{O}_{2}}{\langle \psi,\varphi\rangle J}=i(\psi,\varphi)+\dim\frac{\mathcal{O}_{2}}{\langle \psi\rangle+J}+\dim\frac{\mathcal{O}_{2}}{J}.
    \end{equation}
\end{lemma}
\begin{proof}
Write $I=\langle\psi,\varphi\rangle$, $d=\nu(\psi)$, and $q=\max\{k-d,0\}$. Since $\nu(\psi)\leq\nu(\varphi)$ and the order of a product is the sum of the orders, we have
\[
J:I=\{a\in\mathcal{O}_2:a\psi,a\varphi\in J\}=\m^q.
\]
Consider the following short exact sequence:
\begin{equation}\label{ses}
0\longrightarrow\frac{\mathcal{O}_2}{\m^q}
\xrightarrow{\ \sigma\ }
\left(\frac{\mathcal{O}_2}{J}\right)^2
\xrightarrow{\ \pi\ }
\frac{I}{IJ}\longrightarrow0,
\end{equation}
where $\sigma(\overline{a})=(-\varphi a,\psi a)$ and $\pi(\overline{u},\overline{v})=\psi u+\varphi v\pmod{IJ}$. The injectivity of $\sigma$ follows from $J:I=\m^q$. To verify exactness at the middle term, suppose that $\psi u+\varphi v=\psi j_1+\varphi j_2$, with $j_1,j_2\in J$. Since $\psi$ and $\varphi$ are relatively prime, there exists $a\in\mathcal{O}_2$ such that $u-j_1=-\varphi a$ and $v-j_2=\psi a$. Thus $\ker\pi=\operatorname{Im}\sigma$.
Taking complex dimensions in (\ref{ses}) and in the short exact sequence
\[
0\longrightarrow\frac{I}{IJ}\longrightarrow\frac{\mathcal{O}_2}{IJ}\longrightarrow\frac{\mathcal{O}_2}{I}\longrightarrow0,
\]
we obtain
\begin{equation}\label{eq-pq}
\dim\frac{\mathcal{O}_2}{IJ}=i(\psi,\varphi)+2\dim\frac{\mathcal{O}_2}{J}-\dim\frac{\mathcal{O}_2}{\m^q}.
\end{equation}
On the other hand, multiplication by $\psi$ gives the short exact sequence
\[
0\longrightarrow\frac{\mathcal{O}_2}{\m^q}\xrightarrow{\ \cdot\psi\ }\frac{\mathcal{O}_2}{J}\longrightarrow\frac{\mathcal{O}_2}{\langle\psi\rangle+J}\longrightarrow0,
\]
since $J:\psi=\m^q$. Therefore,
\begin{equation}\label{eq_op}
\dim\frac{\mathcal{O}_2}{\langle\psi\rangle+J}=\dim\frac{\mathcal{O}_2}{J}-\dim\frac{\mathcal{O}_2}{\m^q}.
\end{equation}
Substituting (\ref{eq_op}) into (\ref{eq-pq}), we conclude the proof.
\end{proof}
\begin{remark}
  In Lemma \ref{ca}, the hypothesis $J=\m^k$ is essential; the formula does not hold for an arbitrary ideal of finite colength. To compute $\dim \frac{\mathcal{O}_{2}}{\langle \psi,\varphi\rangle J}$ using this formula, it is necessary to choose the function that has the smallest order.
\end{remark}
\begin{example}
Using \texttt{Singular} \cite{Singular}, we verified the lemma above.

Let   
\[
f = x^4 - y^3, \quad g = y^5 - x^7 + x^4y^4
\]

\begin{center}
\renewcommand{\arraystretch}{1.2}
\begin{tabular}{|c|c|c|c|c|}
\toprule
$k$ & $\dim \mathcal{O}_2 / (f,g) \mathfrak{m}^k$ 
& $\dim \mathcal{O}_2 / \mathfrak{m}^k$ 
& $\dim \mathcal{O}_2 / (f)+\mathfrak{m}^k$ 
& $\dim \mathcal{O}_2 / (f,g)$ \\
\midrule
0  & 20  & 0  & 0   & 20 \\
1  & 22  & 1  & 1    & 20 \\
2  & 26  & 3  & 3    & 20 \\
3  & 32  & 6  & 6    & 20 \\
4  & 39  & 10 & 9   & 20 \\
5  & 47  & 15 & 12 & 20 \\
6  & 56  & 21 & 15  & 20 \\
7  & 66  & 28 & 18  & 20 \\
8  & 77  & 36 & 21  & 20 \\
9  & 89  & 45 & 24  & 20 \\
10 & 102 & 55 & 27  & 20 \\
\bottomrule
\end{tabular}
\end{center}
\end{example}

\section{Formula for the $k$-th Milnor number of a foliation}\label{fomula_Milnor}

Let $\F : \omega = P\,dx + Q\,dy$ be a singular foliation at $0 \in \C^2$.  
Assume that the algebraic multiplicity of $\F$ satisfies $\nu(\F) = \nu(P) \leq \nu(Q)$.  
By applying Lemma~\ref{ca} with $J = \m^k$, $\psi = P$, and $\varphi = Q$, we obtain
\begin{equation}\label{eq_apli}
    \mu^{k}(\F)
    = \mu(\F)
    + \dim\frac{\mathcal{O}_2}{\m^{k}}
    + \dim\frac{\mathcal{O}_2}{\langle P\rangle + \m^{k}}.
\end{equation}
Since $\dim\dfrac{\mathcal{O}_2}{\m^{k}} = \dfrac{k(k+1)}{2}$, we have, more generally, the following observation:

\begin{remark}\label{Remark1}
    Let $f\in\mathcal{O}_2$ have an isolated singularity at $0\in\C^2$ and suppose that $\nu(f)=m$. Then
    $$
\dim\frac{\mathcal{O}_2}{\mathfrak{m}^{k}+\langle f\rangle}=\begin{cases}
			\frac{k(k+1)}{2}, & \text{if $0\leq k< m$ }\\
            \frac{(k+1)k}{2}-\frac{(k+1-m)(k-m)}{2}, & \text{if $k\geq m$.}
		 \end{cases}
$$
The same formula holds for every nonzero germ $f\in\mathcal{O}_2$, without the assumption of an isolated singularity.
\end{remark}

\noindent As a consequence of equation \ref{eq_apli}, and Remark \ref{Remark1}, we obtain the following corollary:
\begin{corollary}\label{mk}
Let $f\in\mathcal{O}_2$ have an isolated singularity at $0\in\C^2$. Suppose that $\nu(f)=m$. Then
\[\mu^k(f)=\begin{cases}
			\mu(f)+k(k+1), & \text{if $0\leq k< m$ }\\
            \mu(f)+k(k+1)-\frac{(k-m+2)(k-m+1)}{2}, & \text{if $k\geq m$.}
		 \end{cases}\]
\end{corollary}
\par Note that, in order to apply equation (\ref{eq_apli}) and prove Corollary \ref{mk}, we must choose a coordinate system where $\nu(f_x)=\nu(f)-1\leq\nu(f_y)$ or $\nu(f_y)=\nu(f)-1\leq\nu(f_x)$.

\par When $f$ is a quasi-homogeneous polynomial, we have the following result from \cite[Proposition 2.6]{k-Milnor}:
\begin{proposition}\label{prop_hussain}
    For an isolated singularity defined by a weighted homogeneous polynomial $f\in\C[x,y]$, we have
    \[\mu^k(f)=\tau^k(f)+\frac{k(k-1)}{2}.\]
\end{proposition}

\noindent As an application of Corollary \ref{mk} and Proposition \ref{prop_hussain}, we recover Theorems $A$ and $B$ from \cite{k-Milnor}.

\medskip
\noindent Now we present an application for singular foliations. 
\begin{theorem}
    Let $\F$ be a singular foliation at $0\in\C^2$. Suppose that $\F$ is non-dicritical and of second type. Let $C=\{f=0\}$ be the total union of separatrices at $0\in\C^2$. Then, assuming that $\nu(f_x)\leq\nu(f_y)$, we have 
    \[
\mu^{k}(\F)-\mu^{k}(C)=\mu(\F)-\mu(C).
\]
In particular, $\F$ is a generalized curve foliation if and only if $\mu^{k}(\F)=\mu^{k}(C)$ for all $k\geq 0$.
\end{theorem}

\begin{proof}
Assume that $\F$ is non-dicritical defined by $\omega=Pdx+Qdy$ such that $\nu(\F)=\nu(P)$, and let $C=\{f=0\}$ be the total union of all separatrices at $0\in\C^2$.  
Since $\nu(f_x)\leq\nu(f_y)$, we obtain from Lemma \ref{ca} that
\begin{equation}\label{eq_apli1}
\mu^{k}(C)=\mu(C)+\dim\frac{\mathcal{O}_2}{\m^{k}}+\dim\frac{\mathcal{O}_2}{\langle f_x\rangle+\m^{k}}.
\end{equation}
As $\F$ is of second type, using Proposition \ref{prop:Equa-Ba} we have 
$\nu(\F)=\nu(C)-1=\nu(f_x)$. Hence, it follows from Remark \ref{Remark1} that
\[
\dim\frac{\mathcal{O}_2}{\langle P\rangle+\m^{k}}=\dim\frac{\mathcal{O}_2}{\langle f_x\rangle+\m^{k}}.
\]
Therefore, from (\ref{eq_apli}) and (\ref{eq_apli1}) we obtain
\[
\mu^{k}(\F)-\mu^{k}(C)=\mu(\F)-\mu(C).
\]
Consequently, $\F$ is a generalized curve foliation if and only if $\mu^{k}(\F)=\mu^{k}(C)$ for all $k\geq 0$, by Theorem 7 in \cite{CLS}.
\end{proof}

\section{On the topological invariance of the $k$-th Milnor and Tjurina Numbers}\label{topo}

Given $f\in\mathcal{O}_2$ with an isolated singularity at $0\in\C^2$, it follows from Theorems 2.28 and 2.26 of \cite{Greuel} that $\mu^k(f)$ and $\tau^k(f)$ are \textit{analytic invariants}.  In what follows, we will see that, in general, the $k$-th Tjurina number is not a topological invariant. Specifically, we focus on $\tau^1(f)$.
\par To see this, let us consider $f(x,y)=y^3-x^7$ and $g(x,y)=f+x^5y$. According to \cite[p. 218]{Greuel}, the curves $C=\{f=0\}$ and $D=\{g=0\}$ are topologically equivalent, that is, there exists a homeomorphism $\phi:(\C^2,0)\to(\C^2,0)$ such that $\phi(C)=D$. However, using \textsc{Singular} \cite{Singular}, we obtain:
\begin{verbatim}
> ring r=0,(x,y),ds;      
> LIB "sing.lib";          
> poly f=y3-x7;               
> poly g=f+x5y;                 
> ideal m=x,y;                                     
> vdim(std(tjurina(f)));                         
12                                        
> vdim(std(tjurina(g)));                             
11                                            
> ideal J1=jacob(f)*m+f;               
> ideal J2=jacob(g)*m+g;                 
> vdim(std(J1));   // 1-th Tjurina number of f              
14                                          
> vdim(std(J2));  // 1-th Tjurina number of g               
13      
\end{verbatim}

This shows that $\tau^{1}(f)=14\neq 13=\tau^{1}(g)$ and, therefore, $\tau^{1}(f)$ is not a topological invariant.

\par Concerning the $k$-th Milnor number of a germ $f\in\mathcal{O}_2$ with an isolated singularity at $0\in\C^2$ and $\nu(f)=m\geq 2$, Corollary \ref{mk} states that
\[
\mu^k(f)=\begin{cases}
			\mu(f)+k(k+1), & \text{if $0\leq k< m$ }\\[4pt]
            \mu(f)+k(k+1)-\frac{(k-m+2)(k-m+1)}{2}, & \text{if $k\geq m$.}
		 \end{cases}
\]
This shows that, in dimension two, $\mu^{k}(f)$ depends only on $\mu(f)$, $k$, and $\nu(f)$. Since $\mu(f)$ is a topological invariant (see \cite{Le}) and $\nu(f)$ is also a topological invariant (see \cite{Oscar}), it follows that $\mu^{k}(f)$ is a topological invariant as well.  In higher dimensions, the problem of the topological invariance of $\mu^{k}(f)$ remains open. 

\par Regarding the topological invariance of the $k$-th Milnor number of a foliation, we can say the following:
Let $\F:\omega=Pdx+Qdy$ be a germ of foliation at $0\in\C^2$. Let us assume that $\nu(\F)=\nu(P)\leq\nu(Q)$. Applying Lemma \ref{ca}, we obtain
\begin{equation*}
    \mu^{k}(\F)=\mu(\F)+\dim\frac{\mathcal{O}_2}{\m^k}+\dim\frac{\mathcal{O}_2}{\langle P\rangle+\m^k}.
\end{equation*}
As a consequence, $\mu^{k}(\F)$ depends on $\mu(\F)$, $k$, and the algebraic multiplicity $\nu(\F)=\nu(P)$ (by hypothesis). 
Since it is known that the algebraic multiplicity $\nu(\F)$ is a biLipschitz invariant (see \cite{Rudy}) and $\mu(\F)$ is a topological invariant (see \cite{CLS}), we can conclude that $\mu^{k}(\F)$ is a biLipschitz invariant, that is, it is preserved under a Lipschitz equivalence. 

\par The problem of the topological invariance of both the algebraic multiplicity $\nu(\F)$ and $\mu^{k}(\F)$ remains open. We should mention that there is a recent pre-print claiming that the algebraic multiplicity 1 is a topological invariant of the foliation (see \cite{Camara}).

\section{On the conjecture for the $k$-th Tjurina number of a weighted homogeneous polynomial}\label{conjecture_lower}

In \cite[Conjecture 1.1]{k-Milnor}, the authors conjectured the following sharp lower bound for the $k$-th Tjurina number of a complex hypersurface:

\noindent{\bf Conjecture 1.}
For each $k \geq 0$, assume that $\tau^{k}(\{x_1^{a_1}+\ldots+x_n^{a_n}=0\})=\ell^k(a_1,\ldots,a_n)$.  
Let $(V,0)=\{(x_1,\ldots,x_n)\in\C^n : f(x_1,\ldots,x_n)=0\}$, with $n\geq 2$, be an isolated singularity defined by a weighted homogeneous polynomial $f(x_1,\ldots,x_n)$ of weight type $(w_1,\ldots,w_n;1)$. Then
\[
\tau^{k}(V)\geq\ell^{k}(1/w_1,\ldots,1/w_n).
\]

In this section, we provide a \textit{positive answer to Conjecture~1} in the case $n=2$.  
Indeed, let $C=\{f(x,y)=0\}$ be a germ of a reduced singular curve defined by $f\in\mathcal{O}_2$.  
Assume that $f$ is a weighted homogeneous polynomial of weight type $(1/w_1,1/w_2;1)$ and let $m=\nu(f)\geq\frac{1}{w_1}\geq 2$.  
Then, it follows from Corollary~\ref{mk}, Proposition~\ref{prop_hussain}, and the Milnor–Orlik formula \cite{Orlik} that
\[
\tau^k(f)=
\begin{cases}
\left(\frac{1}{w_1}-1\right)\left(\frac{1}{w_2}-1\right)+\frac{k(k+3)}{2}, & \text{if } 0\leq k< m, \\[4pt]
\left(\frac{1}{w_1}-1\right)\left(\frac{1}{w_2}-1\right)+\frac{k(k+3)}{2}-\frac{(k-m+2)(k-m+1)}{2}, & \text{if } k\geq m.
\end{cases}
\]

Similarly, we have the following formula for $\ell^{k}(1/w_1,1/w_2)$:
\[
\ell^{k}(1/w_1,1/w_2)=
\begin{cases}
\left(\frac{1}{w_1}-1\right)\left(\frac{1}{w_2}-1\right)+\frac{k(k+3)}{2}, & \text{if } 0\leq k< 1/w_1, \\[4pt]
\left(\frac{1}{w_1}-1\right)\left(\frac{1}{w_2}-1\right)+\frac{k(k+3)}{2}-\frac{\big(k-\frac{1}{w_1}+2\big)\big(k-\frac{1}{w_1}+1\big)}{2}, & \text{if } k\geq 1/w_1.
\end{cases}
\]

These formulas imply that, in order to prove $\tau^k(f)\geq \ell^{k}(1/w_1,1/w_2)$, it suffices to show that for $m\geq 1/w_1$, the following inequality holds:
\begin{equation}\label{eq_conjetu}
    \frac{\big(k-\frac{1}{w_1}+2\big)\big(k-\frac{1}{w_1}+1\big)}{2}\geq \frac{(k-m+2)(k-m+1)}{2}.
\end{equation}

Note that if $m=\frac{1}{w_1}$, the inequality is an equality, and there is nothing to prove.  
Let us therefore assume that $m>1/w_1$ and prove inequality~(\ref{eq_conjetu}).  
Multiplying and eliminating common factors in~(\ref{eq_conjetu}), we see that it is equivalent to
\[
2k\left(m-\frac{1}{w_1}\right)\geq \left(m-\frac{1}{w_1}\right)\left(m+\frac{1}{w_1}-3\right).
\]
Since $m>1/w_1$, we can simplify the factor $\left(m-\frac{1}{w_1}\right)$ and obtain
\[
2k\geq \left(m+\frac{1}{w_1}-3\right).
\]
This last inequality is clearly true because we are considering $k\geq m>1/w_1$.  
Taking into account the above computations, we can now state the following theorem.

\begin{theorem}
Let $C$ be a germ of a reduced singular curve at $0\in\C^2$ defined by a weighted homogeneous polynomial $f$ with an isolated singularity at $0\in\C^2$, of weight type $(w_1,w_2;1)$ such that $\nu(f)\geq 1/w_1$. Then
\[
\tau^{k}(C)\geq \ell^k\left(\frac{1}{w_1},\frac{1}{w_2}\right).
\]
\end{theorem}

This result generalizes Theorem~C of \cite{k-Milnor}.

\section{On a Conjecture for the ratio $\mu^{k}/\tau^k$}\label{Conjecture}

In \cite{Dimca-Greuel}, Dimca and Greuel questioned the relationship between the Milnor and Tjurina numbers of an isolated plane curve singularity. More precisely, they asked whether, for every isolated singularity $(C,0)$ of a reduced plane curve, the inequality
\[
\frac{\mu(C)}{\tau(C)} < \frac{4}{3}
\]
holds. The definitive answer to this problem was given by Almir\'on \cite{Almiron}. 

\par In \cite[Conjecture 1.3]{Hussain_1}, the authors conjectured that, for any isolated plane curve $C$, the following inequality is satisfied:
\[
\frac{\mu^{k}(C)}{\tau^{k}(C)}<\frac{4}{3}\quad\text{for all}\quad k\geq 0.
\]
They considered the curve $C=\{f=0\}$, where
\[
f(x,y)=x^{2m+1}+y^{2m+1}+x^{m+1}y^{m+1},\quad m\geq 1,
\]
and showed that $\frac{\mu^{1}(C)}{\tau^{1}(C)}$ and $\frac{\mu^{2}(C)}{\tau^{2}(C)}$ are sufficiently close to $\frac{4}{3}$ when $m$ is large enough.

\par However, we claim that such a conjecture has a \textit{negative answer in general}. Indeed, consider $m=2$. In this case we have $f=x^5+y^5+x^3y^3$. Using the \textsc{Singular} code \cite{Singular}:

\begin{verbatim}
> ring r=0,(x,y),ds;
> poly f=x5+y5+x3y3;
> ideal m=x,y;
> ideal J=jacob(f)*m*m*m*m*m*m;
> ideal I=J+f;
> vdim(std(J));  // 6-th Milnor number
55
> vdim(std(I));  // 6-th Tjurina number 
40
\end{verbatim}

Note that
\begin{equation}\label{ex_conjectura}
\frac{\mu^{6}(C)}{\tau^{6}(C)}=\frac{55}{40}>\frac{4}{3}.
\end{equation}
\par An analogous computation, using Singular, shows that $\mu^{5}(C)/\tau^{5}(C)=45/35<4/3$ and therefore, level 6 is the first quotient that contradicts the conjecture in question.
\par
Next, we confirm Conjecture 1.3 of \cite{Hussain_1} for quasi-homogeneous polynomials with a restriction on $k$.

\begin{theorem}\label{teo_conje}
Let $C=\{ f=0\}$ be an isolated singularity defined by a quasi-homogeneous polynomial $f$ of algebraic multiplicity $\nu(f)$. Then, for every $0\leq k<\nu(f)$, we have
\[
\frac{\mu^{k}(f)}{\tau^{k}(f)}<\frac{4}{3}.
\]
\end{theorem}

\begin{proof}
Since $0\leq k<\nu(f)$, by Corollary \ref{mk} we have $\mu^{k}(f)=\mu(f)+k(k+1)$. Moreover, in the quasi-homogeneous case, by Proposition \ref{prop_hussain}, it follows that 
\[
\mu^{k}(f)=\tau^{k}(f)+\frac{k(k-1)}{2}.
\]
Using the valid formulas in the quasi-homogeneous case, for \(0\le k<\nu(f)\) we have
\[
\mu^k(f)=\mu(f)+k(k+1),
\qquad
\mu^k(f)=\tau^k(f)+\frac{k(k-1)}{2}.
\]
From the second equality, we immediately obtain
\[
\tau^k(f)=\mu^k(f)-\frac{k(k-1)}{2}=\mu(f)+\frac{k^2+3k}{2}.
\]
Thus, for \(0\le k<\nu(f)\), the desired quotient is
\[
\frac{\mu^k(f)}{\tau^k(f)}
= \frac{\mu(f)+k(k+1)}{\mu(f)+\dfrac{k^2+3k}{2}}.
\]
To verify that this quotient is strictly less than \(4/3\), it is enough to check the equivalent inequality
\[
3\left(\mu(f)+k(k+1)\right) < 4\left(\mu(f)+\frac{k^2+3k}{2}\right).
\]
Simplifying, we obtain the condition
\[
\mu(f) > k^2-3k.
\tag{*}
\]
By observation in \cite[Remark 7.4]{Arturo}, we have \((\nu(f)-1)^2\le \mu(f)\). We now show that, for every integer \(k\) with \(0\le k<\nu(f)\),
\[
(\nu(f)-1)^2 > k^2-3k,
\]
which, together with \((\nu(f)-1)^2\le\mu(f)\), implies \((*)\) and hence the desired inequality.

Indeed,
\[
(\nu(f)-1)^2 - (k^2-3k)
= (\nu(f)-1-k)(\nu(f)-1+k) + 3k.
\]
Since \(0\le k<\nu(f)\), we have \(\nu(f)-1-k\ge 0\) and \(3k\ge 0\). Moreover, if \(k>0\), then \(3k>0\), so the difference is strictly positive; if \(k=0\), then \((\nu(f)-1)^2>0\) for \(\nu(f)\ge 2\), and we also obtain the inequality. Hence
\[
(\nu(f)-1)^2 > k^2-3k \quad\text{for all }0\le k<\nu(f),
\]
and therefore \(\mu(f)\ge(\nu(f)-1)^2>k^2-3k\), which is exactly the condition (\(*\)). Thus, for all \(0\le k<\nu(f)\),
\[
\frac{\mu^k(f)}{\tau^k(f)}<\frac{4}{3}.
\]
\end{proof}

\par This result generalizes \cite[Theorem E]{Hussain_1}, since it shows that the inequality 
\[
\frac{\mu^{k}(f)}{\tau^{k}(f)} < \frac{4}{3}
\]
remains valid not only for $k=2$, but also for all values $0\leq k<\nu(f)$ when $f$ is a quasi-homogeneous polynomial.

\section{GSV-index and its relation with the $k$-th Tjurina numbers}\label{GSV-section}

Let $\F$ be a foliation on $(\mathbb{C}^{2},0)$ defined by the 1-form $\omega=P(x,y)dx +Q(x,y)dy$, and let $C=\{f=0\}$ be a germ of a reduced invariant curve, with $f\in\mathcal{O}_2$. Then, there exist elements $g, h \in \mathcal{O}_2$, depending on $f$ and $\omega$, such that $f$ is relatively prime to both $g$ and $h$, and a 1-form $\eta$ satisfying
\begin{equation} \label{eq_3}
    g\omega = h\,df + f\,\eta,
\end{equation}
as established in \cite[Lemma~1.1]{Suwa}. The GSV-index of a foliation $\mathcal{F}$ with respect to the curve $C$ at the point $0\in\C^2$ is defined by
\begin{equation*} \label{eq:GSVintegral}
    \operatorname{GSV}(\mathcal{F}, C) = \frac{1}{2\pi i} \int_{\partial C} \frac{h}{g} \, d\left( \frac{g}{h} \right),
\end{equation*}
where $g$ and $h$ are chosen according to equation \eqref{eq_3}. Alternatively, following \cite{Evelia},
\begin{equation}\label{gsv_inter}
\operatorname{GSV}(\mathcal{F},C)=i(f,h)-i(f,g),
\end{equation}
where \(g,h\in\mathcal{O}_2\) are the elements mentioned in \eqref{eq_3} and  \(C:\;f(x,y)=0\) is a reduced curve.

The following proposition was first established by G\'omez-Mont \cite[Theorem 1]{Gomez}. A proof in dimension 2 can be found in \cite[Proposition 6.2]{Evelia}:  
\begin{proposition}\label{prop_gomez}
    Let \(\mathcal{F}\) be a singular foliation on \((\mathbb{C}^2, 0)\), and \(C\) a reduced curve of separatrices of \(\mathcal{F}\). Then
\[
\tau(\mathcal{F}, C) - \tau(C) = \operatorname{GSV}(\mathcal{F}, C).
\]
\end{proposition}

In this section, we will relate the GSV-index of $\F$ with the $k$-th Tjurina number of both the curve $C$ and the foliation $\F$ with respect to the curve $C$, obtaining a generalization of Proposition \ref{prop_gomez}:  

\begin{theorem}\label{teoA}
\label{teo tk}
 Let $\F$ be a singular foliation at $(\C^2,0)$, and let $C$ be a reduced curve of separatrices of $\F$. Then
    \[\tau^{k}(\F,C)-\tau^{k}(C)=GSV(\F,C)\quad\text{for all}\quad k\geq 0.\]
\end{theorem}
\begin{proof}
  Let  $\omega=Pdx+Qdy$ be a 1-form inducing $\F$, and let $f(x,y)=0$ be the reduced equation of $C$. By equality (\ref{eq_3}) we get $g\omega=hdf+f\eta$, where $\eta$ is a 1-form and $g,h\in\mathcal{O}_2$, with $g$ and $f$, and $h$ and $f$ relatively prime. 
  \par Hence $gPdx+gQdy=(hf_x+f\eta_1)dx+(hf_y +f\eta_2)dy$, where $\eta=\eta_1 dx+\eta_2 dy$. We obtain 
  \begin{equation}\label{eq_4}
      g P=hf_x+f\eta_1,\quad\text{and}\quad gQ=hf_y +f\eta_2.
  \end{equation}
From equalities (\ref{eq_4}) and the properties of ideals, we have
\[
\langle gP,gQ\rangle\m^k+\langle f\rangle=\langle hf_x,hf_y\rangle\m^k+\langle f\rangle.
\]
By applying Lemma \ref{lemma_gsv 1} to both sides, we obtain
\[
\tau^k(\F,C)+i(f,g)=\tau^k(C)+i(f,h).
\]
Hence, by equation (\ref{gsv_inter}),
\begin{eqnarray*}
\tau^k(\F,C)-\tau^k(C)&=&i(f,h)-i(f,g)\\
&=&GSV(\F,C).
\end{eqnarray*}
\end{proof}
We illustrate Theorem \ref{teoA} with the following examples.
\begin{example}
     Let $\F$ be the foliation on $(\mathbb{C}^{2},0)$ defined by $\omega = 4xydx + (y - 2x^{2})dy$. Note that $C = \{y = 0\}$ is the only separatrix curve of $\F$ and that an equation (\ref{eq_3}) for $C$ is $g = 1$, $h = y - 2x^{2}$, and $\eta = 4xdx$. Thus $GSV(\F,C)=2$, $\tau^k(C)=k$, and $\tau^{k}(\F,C)=k+2$. Hence, we have $GSV(\F,C)= 2=\tau^{k}(\F, C) - \tau^{k}(C).$
\end{example}
\begin{example}
    For the Dulac foliation $\F$ defined by \[\omega=(ny+x^{n})dx-xdy,\quad n\geq 2,\] we have that $\F$ admits a unique separatrix $C=\{x=0\}$. Since $\nu(\F)=1\neq 0=\nu(C)-1$, it follows that $\F$ is not of second type.
    Moreover, for every integer $k\geq 0$, we get
    \[
    \mu^{k}(\F)=\frac{(k+1)(k+2)}{2},\qquad
    \tau^{k}(\F,C)=k+1,\qquad
    \tau^{k}(C)=k,
    \]
    and
    \[
    GSV(\F,C,0)=1=\tau^{k}(\F, C) - \tau^{k}(C).
    \]
\end{example}
\begin{example}
    Let $\F_{n,m}:\omega_{n,m}=mxdy-nydx$, and let $C=\{y^m-x^n=0\}$ be a separatrix of $\F_{n,m}$, where $2\leq m\leq n$. 
    We have from \cite[Theorem A, item 1]{k-Milnor} that
    \[
    \tau^{k}(C)=
    \begin{cases}
        \dfrac{k^2+3k}{2}+mn-m-n+1, & \text{if } 0\leq k< m,\\[0.5em]
        mk+\dfrac{(2n-m)(m-1)}{2}, & \text{if } k\geq m,
    \end{cases}
    \]
    \[
    \mu^{k}(C)=
    \begin{cases}
        mn-m-n+1+k^2+k, & \text{if } 0\leq k< m,\\[0.5em]
        \left(m-\dfrac{1}{2}\right)k+\dfrac{(2n-m)(m-1)}{2}+\dfrac{k^2}{2}, & \text{if } k\geq m,
    \end{cases}
    \]
    and 
    \[
    \tau^{k}(\F_{n,m},C)=
    \begin{cases}
        \dfrac{(k+2)(k+1)}{2}, & \text{if } 0\leq k< m,\\[0.5em]
        \dfrac{(k+2)(k+1)}{2}-\dfrac{(k+2-m)(k+1-m)}{2}, & \text{if } k\geq m.
    \end{cases}
    \]
    On the other hand, since 
    \[
    my^{m-1}\cdot\omega = mx\,d(y^m-x^n)-mn(y^m-x^n)\,dx,
    \]
    we get
    \[
        GSV(\F_{n,m},C)=m+n-mn.
    \]
    Note that $GSV(\F_{n,m},C)=\tau^{k}(\F_{n,m},C)-\tau^{k}(C)$ for all $k\geq 0$, verifying Theorem \ref{teo tk}.
\end{example}
\begin{corollary}
    Let $\F$ be a singular foliation at $(\C^2,0)$, and let $C$ be a reduced curve of separatrices of $\F$. Then
     \[\tau^{k}(\F,C)-\tau^{k}(C) = \tau(\F,C)-\tau(C)\quad\text{for all}\quad k\geq 0,\] 
     or equivalently,
     \[\tau^{k}(\F,C)-\tau(\F,C) = \tau^{k}(C)-\tau(C)\quad\text{for all}\quad k\geq 0.\] 
     In particular, $\tau^{k}(\F,C)-\tau(\F,C)$ does not depend on the foliation $\F$.
\end{corollary}
\begin{proof}
   It is an immediate consequence of Proposition \ref{prop_gomez} and Theorem \ref{teo tk}.
\end{proof}
\par The $k$-th Tjurina number of the foliation $\mathcal{F}$ with respect to a separatrix curve $C$ yield the following criterion for dicriticalness.
\begin{corollary}
    Let $C$ be a separatrix curve of $\mathcal{F}$, and suppose that $\tau^{k}(\mathcal{F}, C) < \tau^{k}(C)$ for some $k \geq 0$. Then $\mathcal{F}$ is dicritical at $0 \in \mathbb{C}^{2}$.
\end{corollary}
\begin{proof}
    Suppose, by contradiction, that $\F$ is non-dicritical at $0\in\C^2$. Then $C$ is a non-dicritical separatrix in the sense of Brunella (see \cite[p. 533]{index}), and by \cite[Proposition 6]{index}, we have $GSV(\F,C)\geq 0$. Hence $\tau^{k}(\mathcal{F}, C) \geq \tau^{k}(C)$ for all $k \geq 0$, by Theorem \ref{teoA}. This contradicts the hypothesis.  
\end{proof}

\section{The $k$-th polar intersection index of a foliation}\label{polar_section}
Let $\F$ be a singular foliation on $(\C^2,0)$ defined by $\omega=P(x,y)\,dx+Q(x,y)\,dy$.
\begin{definition}
Let $[\alpha:\beta ]\in \mathbb{P}^1$ be a direction in the one-dimensional projective space.  
The \emph{polar curve} of $\mathcal{F}$ in the direction $[\alpha:\beta]$, denoted by $\Gamma_{[\alpha:\beta]}$, is the set of points defined by
\[
\mathcal{P}^{\mathcal{F}}_{(\alpha:\beta)}= \{(x,y)\in(\mathbb{C}^2,0)\;:\;\alpha P(x,y)+\beta Q(x,y)=0\}.
\]
\end{definition}
The polar curve of a foliation was defined by N. Corral in \cite{Corral1}, see also \cite{Corral2}, \cite{Corral3}, and \cite{Cano}. Observe that the algebraic multiplicity at the origin of $\mathbb{C}^2$ of a generic polar curve coincides with the algebraic multiplicity of $\mathcal{F}$ at that point.
Also note that when $\mathcal{F}$ is the Hamiltonian foliation $df$ associated with a function $f\in\mathcal{O}_2$, the polar curve of $\mathcal{F}$ coincides with the classical polar curve of $f$ in the direction $[\alpha:\beta]$ studied by Teissier \cite{Teissier_inv} and others. According to general results on equisingularity (see \cite{Zariski} and \cite{Teissier_polar}), there exists a Zariski-open subset $U$ of the space $\mathbb{P}^1$ of projection directions such that, for $[\alpha:\beta]\in U$, all polar curves are equisingular. Any element of this set is called the {\it generic polar curve} of the foliation $\mathcal{F}$ and will be denoted by $\mathcal{P}^{\mathcal{F}}$.
\begin{definition}\label{defi_polar}
Let $B =\{h = 0\}$ be a reduced curve invariant by a singular foliation $\mathcal{F}$.  
The \emph{polar intersection number} of $\mathcal{F}$ with respect to $B$ is defined as the intersection number
\[
i_p(\mathcal{P}^{\mathcal{F}}, B)= \dim\frac{\mathcal{O}_2}{\langle \alpha P+\beta Q,\,h\rangle}
\]
where $\mathcal{P}^{\mathcal{F}}$ denotes a generic polar curve associated with the foliation $\mathcal{F}$.
\end{definition}
For properties of this number, see for instance \cite[Section 4]{mol2}.
\begin{definition}
Let $C=\{f=0\}$ be a reduced invariant curve of a singular foliation $\mathcal{F}$. We define the {\it $k$-th polar intersection number} of $\mathcal{F}$ with respect to $C$ as the following dimension:
\[
i^{k}(\mathcal{P}^{\mathcal{F}},C)=\dim\frac{\mathcal{O}_2}{\langle \alpha P+\beta Q,\,f\rangle \m^{k}},
\]
where $\mathcal{P}^{\mathcal{F}}$ denotes a generic polar curve associated with the foliation $\mathcal{F}$.
\end{definition}
Note that when $k=0$, we recover the classical polar intersection number of a foliation $\mathcal{F}$ with respect to $C$, given in Definition \ref{defi_polar}. 

\subsection{Teissier's Lemma for the $k$-th Milnor number} 
 In the case of the Hamiltonian foliation defined by $df=f_x dx+f_y dy$, we have:
\[
i^{k}(\mathcal{P}^{df},C)=\dim\frac{\mathcal{O}_2}{\langle \alpha f_x+\beta f_y,\,f\rangle\m^{k}}.
\]
which generalizes the polar intersection number defined by Teissier \cite{Teissier}. In \cite[Chapter 4, Proposition 1.2]{Teissier}, Teissier proved (in the context of complex hypersurfaces) a formula relating $i(\mathcal{P}^{df},C)$ and $\mu(C)$, known as Teissier's Lemma:
\begin{lemma}[Teissier's Lemma]\label{lemma_tei}
    Let $f\in\mathcal{O}_2$ be such that $C=\{f=0\}$ is a reduced curve at $0\in\C^2$. Then
\[i(\mathcal{P}^{df},C)=\mu(C)+\nu(f)-1.\]
\end{lemma}

Next, we establish the $k$-th version of Teissier's Lemma \cite[Chapter 4, Proposition 1.2]{Teissier} in dimension 2:
\begin{lemma}
Let $f\in\mathcal{O}_2$ be such that $C=\{f=0\}$ is a reduced curve at $0\in\C^2$. Then, for all integers $k\geq 0$, we have 
\[i^{k}(\mathcal{P}^{df},C)=\mu^{k}(C)+\nu(f)-1.\]
\end{lemma}
\begin{proof}
By applying Lemma \ref{ca} to $J=\m^{k}$, $\psi=\mathcal{P}^{f}=\alpha f_x+\beta f_y$, where $\alpha$ and $\beta$ are generic, and $\varphi=f$, we obtain 
\begin{equation}\label{eq_30}
i^k(\mathcal{P}^{df},C)=i(\mathcal{P}^{df},C)+\dim\frac{\mathcal{O}_2}{\m^k}+\dim\frac{\mathcal{O}_2}{\langle \alpha f_x+\beta f_y\rangle+\m^k}
\end{equation}
Again, assuming $\nu(f_x)\leq\nu(f_y)$, by applying Lemma \ref{ca} to $J=\m^{k}$, $\psi=f_x$, and $\varphi=f_y$, we obtain 
\begin{equation}\label{eq_40}
\mu^{k}(C)=\mu(C)+\dim\frac{\mathcal{O}_2}{\m^k}+\dim\frac{\mathcal{O}_2}{\langle f_x\rangle+\m^k}
\end{equation}
From (\ref{eq_30}) and (\ref{eq_40}), we get:
\[i^k(\mathcal{P}^{df},C)-\mu^{k}(C)=i(\mathcal{P}^{df},C)-\mu(C)\]
since $\dim\dfrac{\mathcal{O}_2}{\langle \alpha f_x+\beta f_y\rangle+\m^k}=\dim\dfrac{\mathcal{O}_2}{\langle f_x\rangle+\m^k}$ for generic $\alpha$ and $\beta$, by Remark \ref{Remark1}. The proof follows from Lemma \ref{lemma_tei}.
\end{proof}
We now present an example illustrating our result above.

\begin{example}
    Let $C$ be an isolated binomial singularity defined by $f=x^2+y^2$. Then, by \cite[Proposition 2.7]{k-Milnor},
    \[
    \mu^{k}(C)= \dim\frac{\mathcal{O}_2}{\langle f_x,f_y\rangle\m^{k}}
    =
    \begin{cases}
        1+k^{2}+k, & \text{if } 0\leq k< 2,\\[0.5em]
        \dfrac{3}{2}k+1+\dfrac{k^{2}}{2}, & \text{if } k\geq 2,
    \end{cases}
    \]
    and the $k$-th polar intersection number of $df$ with respect to $C$ is 
    \[
    i^{k}(\mathcal{P}^{df},C)= \dim\frac{\mathcal{O}_2}{\langle f_x+ f_y, f\rangle\m^k}
    =
    \begin{cases}
        2+k^{2}+k, & \text{if } 0\leq k< 2,\\[0.5em]
        \dfrac{3}{2}k+2+\dfrac{k^{2}}{2}, & \text{if } k\geq 2,
    \end{cases}
    \]
    thus, $i^{k}(\mathcal{P}^{df},C)-\mu^{k}(C)=\nu (C)-1$.
\end{example}
Now, we present a formula that relates the $k$-th intersection indices of a foliation $\F$ and an invariant curve $C=\{f=0\}$ of $\F$ with the index $GSV(\F,C)$.
\begin{theorem}
    Let $\F$ be a non-dicritical foliation on $(\C^2,0)$, and let $C=\{f=0\}$ be the total union of separatrices of $\F$. If $\F$ is of second type, then
    \[i^k(\mathcal{P}^{\F},C)-i^k(\mathcal{P}^{df},C)=i(\mathcal{P}^{\F},C)-i(\mathcal{P}^{df},C)=GSV(\F,C).\]
\end{theorem}
\begin{proof}
Suppose that $\omega=Pdx+Qdy$. By applying Lemma \ref{ca} to $J=\m^{k}$, $\psi=\mathcal{P}^{\F}=\alpha P+\beta Q$, and $\varphi=f$, we have
\[i^k(\mathcal{P}^{\F},C)=i(\mathcal{P}^{\F},C)+\dim\frac{\mathcal{O}_2}{\m^{k}}+\dim\frac{\mathcal{O}_2}{\langle \alpha P+\beta Q\rangle+\m^{k}},\]
where $\alpha$ and $\beta$ are generic. Similarly, by Lemma \ref{ca}, we obtain
\[i^k(\mathcal{P}^{df},C)=i(\mathcal{P}^{df},C)+\dim\frac{\mathcal{O}_2}{\m^{k}}+\dim\frac{\mathcal{O}_2}{\langle \alpha f_x+\beta f_y\rangle+\m^{k}}.\]
Since $\F$ is of second type, then $\nu(\F)=\nu(\alpha P+\beta Q)=\nu(C)-1=\nu(\alpha f_x+\beta f_y)$, by Proposition \ref{prop:Equa-Ba}. Consequently 
\[\dim\frac{\mathcal{O}_2}{\langle \alpha P+\beta Q\rangle+\m^{k}}=\dim\frac{\mathcal{O}_2}{\langle \alpha f_x+\beta f_y\rangle+\m^{k}},\]
by Remark \ref{Remark1}. The proof ends by taking the difference $i^k(\mathcal{P}^{\F},C)-i^k(\mathcal{P}^{df},C)$ and observing that $i(\mathcal{P}^{\F},C)-i(\mathcal{P}^{df},C)=GSV(\F,C)$ (see \cite[Proposition 5.4]{Evelia}).
\end{proof}

\section{A bound for the $k$-th Milnor number of a foliation}\label{bound_section}
We now present the following result. The proof is an adaptation of \cite[Theorem 1.1]{Liu} and \cite[Theorem 1.3]{Hussain_1}.
\begin{theorem}
    Let $\F$ be a singular holomorphic foliation of second type at $0\in\C^2$ given by the 1-form $\omega=P(x,y)dx+Q(x,y)dy$. Let $\B=\B_0-\B_{\infty}$ be a balanced divisor of separatrices for $\F$ with $\B_0=\{f(x,y)=0\}$. Then
    \[\tau^{k}(\F,\B_0)\leq \mu^{k}(\F)\leq 2\tau^{k}(\F,\B_0)+\frac{k(k+1)}{2}.\]
\end{theorem}
\begin{proof}
We only need to show that $\mu^{k}(\F)\leq 2\tau^{k}(\F,\B_0)+\frac{k(k+1)}{2}$. Let us consider the following exact sequence:
\[0\to Ker(f)\to M^{k}_{\F}\xrightarrow{f} M^{k}_{\F}\to T^{k}_{\F,\B_0}\to 0,\]
where $M^{k}_{\F}=\mathcal{O}_2/\m^{k}\langle P,Q\rangle$, $T^{k}_{\F,\B_0}=\mathcal{O}_2/(\m^{k}\langle P,Q\rangle+\langle f\rangle)$, the middle map is multiplication by $f$, and $Ker(f)$ is the kernel of this map.
\par  According to the Brian\c{c}on-Skoda theorem for foliations \cite[Theorem A]{Skoda}, we have $f^2\in\langle P,Q\rangle$, and therefore $(f^2)\m^k=0$ in $M^{k}_{\F}$. As a consequence, we obtain a finite decreasing filtration:
\[M^{k}_{\F}\supset (f)\supset (f)\m^k\supset (f^2)\m^k=0.\]
Now consider the following exact sequence:
\[0\to Ker(f)\cap (f)\m^k\to (f)\m^k\xrightarrow{f} (f)\m^k\to (f)\m^k/(f^2)\m^k\to 0,\]
where the middle map is multiplication by $f$. Then
\[\dim_\C\left\{(f)\m^k/(f^2)\m^k \right\}=\dim_{\C}(Ker(f)\cap (f)\m^k)\leq \dim_\C Ker(f)=\tau^{k}(\F,\B_0,0).\] Thus
\begin{eqnarray*}
    \mu^{k}(\F)&=&\dim_\C M^{k}_{\F}=\dim_\C T^k_{\F,\B_0}+\dim_\C \frac{(f)}{(f)\m^k}+\dim_{\C}\frac{(f)\m^k}{(f^2)\m^k}\\
    &\leq& 2\tau^{k}(\F,\B_0)+\frac{k(k+1)}{2}.
\end{eqnarray*}
\end{proof}

\section{On $k$-th Milnor and Tjurina numbers for quasi-homogeneous foliations}\label{section_quasi}

In this section, we consider quasi-homogeneous foliations as defined in \cite[p.~178]{David}. Given a germ of a singular holomorphic foliation $\F$ defined by $\omega = P\,dx + Q\,dy$, assume that $\F$ is non-dicritical and let $C = \{ f = 0 \}$ be the total union of separatrices of $\F$ at $0 \in \C^2$.  We say that $\F$ is a \emph{quasi-homogeneous foliation} if $f \in \langle P, Q \rangle$.

\par We prove the following:
\begin{theorem}
Let $\mathcal{F}$ be a non-dicritical generalized curve foliation defined by $\omega = P(x,y)\,dx + Q(x,y)\,dy$ at $0\in\C^2$. Suppose that $\F$ is a quasi-homogeneous foliation. Let \( C = \{f(x,y) = 0\} \) be the total union of separatrices of $\F$ at $0\in\C^2$. Then,  
\[
    \mu^{k}(\mathcal{F}) = \tau^{k}(\mathcal{F},C) + \frac{k(k-1)}{2},\quad \forall\, k\geq 1.
\]
\end{theorem} 

\begin{proof}
First, we show that $\m^{k-1}\langle f\rangle\subset \m^k\langle P,Q\rangle$.
Indeed, by Mattei’s theorem \cite[Th\'eor\`eme~A]{quasiMattei}, there exists a coordinate system $(x, y)$ such that, in this system, $f$ is a quasi-homogeneous polynomial of degree $d$ with weights $(\alpha, \beta)$, and there exist nonzero $g \in \mathcal{O}_2$ and $h \in \mathcal{O}_2$ such that 
\[
g\omega = df + h(\beta y\,dx - \alpha x\,dy),
\]
where $\alpha, \beta \in \mathbb{N}$. Therefore, we have
\begin{eqnarray*}
    \omega = P\,dx + Q\,dy &=& \frac{df + h(\beta y\,dx - \alpha x\,dy)}{g}\\
    &=& \frac{f_x\,dx + f_y\,dy + h(\beta y\,dx - \alpha x\,dy)}{g}\\
    &=& \frac{(f_x+ \beta h y)\,dx}{g} + \frac{(f_y - \alpha h x)\,dy}{g}.
\end{eqnarray*}

Thus, $P = \frac{f_x+ \beta h y}{g}$ and $Q = \frac{f_y - \alpha h x}{g}$.

Without loss of generality, we can assume that $g \equiv 1$, so 
\[
P = f_x+ \beta h y, \quad 
Q = f_y - \alpha h x.
\]
Hence,
\[
\alpha x P = \alpha x f_x+ \alpha\beta x (h y), \qquad 
\beta y Q = \beta y f_y - \beta\alpha y(h x).
\]
Therefore, 
\[
\alpha x P + \beta y Q = \alpha x f_x+ \beta y  f_y,
\]
and by Euler’s identity, we have 
\[
\alpha x P + \beta y Q = d \cdot f 
\Longrightarrow 
\langle f \rangle \subset \m\langle P, Q \rangle 
\Longrightarrow 
\m^{k-1} \langle f \rangle \subset \m^{k} \langle P, Q \rangle, 
\quad \text{for all } k \geq 1.
\]

\par Next, using the short exact sequence
\[
0 \longrightarrow \frac{ \m^{k} \langle P,Q \rangle + \langle f \rangle}{\m^{k} \langle P,Q \rangle} 
\longrightarrow 
\frac{\mathcal{O}_{2}}{\m^{k} \langle P,Q \rangle} 
\longrightarrow 
\frac{\mathcal{O}_{2}}{\m^{k} \langle P,Q \rangle + \langle f \rangle} 
\longrightarrow 0,
\]
we have 
\[
\dim \frac{\mathcal{O}_{2}}{\m^{k} \langle P, Q\rangle}  
= 
\dim  \frac{ \m^{k} \langle P,Q \rangle + \langle f \rangle}{\m^{k} \langle P,Q \rangle}  
+ 
\dim  \frac{\mathcal{O}_{2}}{\m^{k} \langle P,Q \rangle + \langle f \rangle}.
\]  
Therefore, we get
\[
\mu^{k} (\mathcal{F}) = \tau^{k}(\mathcal{F}, C) + \dim \frac{ \m^{k} \langle P,Q \rangle + \langle f \rangle}{\m^{k} \langle P,Q \rangle}.
\]
Next, by applying \cite[Lemma~2.5]{k-Milnor} and using that \( \m^{k-1} \langle f \rangle \subset \m^{k} \langle P,Q \rangle \), we have
\[
\frac{ \m^{k} \langle P,Q \rangle + \langle f \rangle}{\m^{k} \langle P,Q \rangle} 
\cong 
\frac{\langle f \rangle \mathcal{O}_{2}}{\m^{k} \langle P,Q \rangle \cap \langle f \rangle \mathcal{O}_{2}} 
\cong 
\frac{\langle f \rangle \mathcal{O}_{2}}{\m^{k-1} \langle f \rangle} 
\cong 
\frac{\mathcal{O}_{2}}{\m^{k-1}}.
\]
Therefore, we conclude
\[
\mu^{k}(\mathcal{F}) = \tau^{k}(\mathcal{F}, C) +  \frac{k(k-1)}{2}.
\]
\end{proof}

\vspace{2cm}

\noindent{\bf{Funding}}\\
The first author acknowledges support from CAPES-Brazil. The second author acknowledges support from CNPq Projeto Universal 408687/2023-1 ``Geometria das Equa\c{c}\~oes Diferenciais Alg\'ebricas" and the PQ fellowship CNPq-Brazil 306011/2023-9. \\

\noindent{\bf{Acknowledgments}}\\
 The authors are grateful to Marcelo E. Hernandes for his helpful remarks and comments, which helped to improve the presentation of this paper. We also thank Gert-Martin Greuel for bringing to our attention his paper ``Mather–Yau theorem in positive characteristic."

\vspace{1cm}

\noindent{\bf Data Availability Statement:} 
Data sharing is not applicable to this article as no data sets were generated or analyzed during the current study.
\vspace{1cm}

\noindent{\bf Declarations
Conflict of Interest:} The authors declare that they have no conflict of interest.

\end{document}